\documentclass[a4paper]{amsart}

\usepackage{latexsym, amssymb,amsmath, amsthm,amsfonts}
\usepackage{xcolor}

\newcommand{\F}{\mathbb{F}}

\newcommand{\Z}{\mathbb{Z}}
\newcommand{\kk}{\Bbbk}

\def\SL{\operatorname{SL}}

\def\SL2{\operatorname{SL}_{2}(K)}

\def\GL2{\operatorname{GL}_{2}(K)}

\def\GL{\operatorname{GL}}
\def\SL{\operatorname{SL}}
\def\id{\operatorname{id}}

\def\Z{\mathbb{Z}}

\def\im{\operatorname{im}}

\def\res{\operatorname{res}}
\def\im{\operatorname{im}}

\newtheorem{Lemma}{Lemma}
\newtheorem{Theorem}[Lemma]{Theorem}

\newtheorem{Corollary}[Lemma]{Corollary}

\newtheorem{prop}[Lemma]{Proposition}

\newtheorem*{Corollary of Conjecture}{Corollary of Conjecture}

\theoremstyle{definition}

\theoremstyle{remark}
  \newtheorem{rem}[Lemma]{Remark}

\newtheoremstyle{Acknowledgments}
  {}
    {}
     {}
     {}
    {\bfseries}
    {}
     {.5em}
     {\thmname{#1}\thmnumber{ }\thmnote{ (#3)}}
\theoremstyle{Acknowledgments}

\title{The differential invariants of $\SL_2(\F_3)$ acting on trace-free matrices over $\mathbb{F}_3$.}

\author{Jonathan Elmer$^{1}$ \and Anja Meyer$^{2}$}
\date{ \today \\
  $^{1}$Middlesex University \\
  $^{2}$University of Manchester}

\begin{document}

\maketitle

\begin{abstract}
Let $M$ denote the vector space of $2 \times 2$ matrices with coefficients in $\F_3$ and trace zero. Let $G = \SL_2(\F_3)$. Then $G$ acts on $M$ via conjugation. Let $R =(S(M^*) \otimes \Lambda(M^*))$ be the algebra of differential forms on $M$. We compute a minimal generating set for $R^G$ as a commutative-graded algebra. In doing so we utilise the theory of Cohen-Macaulay modules and results in the theory of covariants.
\end{abstract}

\section{Introduction}

    Computing invariant rings can be a challenging task, especially in dividing characteristic. The theory of Cohen-Macaulay modules (e.g. Chapter 2.8 in \cite{CW}) can offer some useful theory to use. Invariant theory plays an important role in group cohomology computations, and we are keen to find new methods for computations. As a first fact we recall that for any finite group $G$ and field $k$, $H^0(G,k)=k^G$. We also observe for a normal subgroup $N$ of $G$, the action of $G/N$ on $N$ induces a linear action of $G/N$ on $H^*(N,k)$. Moving to modular coefficients and using the theory of stable elements, by Chapter II.6 in \cite{AM} $H^*(G,\F_p)\cong H^*(N,\F_p)^{G/N}$. Note that if $N=P$ is further a Sylow $p$-subgroup, this is an application of non-modular invariant theory, as in that case $|G/P|$ is coprime to $|G|$. But even for not necessarily normal subgroups of $G$, invariant theory plays a role: a result of Swan \cite{S} states that for any abelian Sylow $p$-subgroup $P$ of $G$ the composition $ H^*(P,\F_p)^{N_G(P)}\cong \im(\res^G_P) \cong H^*(G,\F_p)$, where the second isomorphism is due to the theory of stable elements as seen in \cite{CE}.  A more specific example is found in \cite{M} where the cohomology ring $H^*(\SL_2(\Z/3^n),\F_3)$, for $n>1$ is computed via stable elements, where the invariant ring $(S(M^*)\otimes \Lambda(M^*))^{\SL_2(\F_3)}$, for $M$ a module of dimension 3, is a main ingredient. This is the invariant ring described in the abstract. In this paper, we compute a minimal set of generators of this graded-commutative invariant ring explicitly, using a combination of Cohen-Macaulay modules and covariants.\\
\noindent We start by giving the necessary background material in modular and non-modular invariant theory, as the computations in this paper combine both: on the one hand, the characteristic of $\F_3$ divides  $|\SL_2(\F_3)|$; on the other hand, we will show that $M$ is isomorphic to a module induced from a non-modular subgroup of $\SL_2(\F_3)$. Then, using a minor generalisation of \cite[Proposition~5]{E} (proved here as Lemma \ref{use molien}) moves us into the non-modular case. This trick enables the use of tools from the non-modular setting, such as the Molien series.\\
\noindent The background section is followed by an explicit computation of the generators of $(S(M^*)\otimes \Lambda(M^*)^{\SL_2(\F_3)}$. These are given in Theorem \ref{Final}, which is the main theorem of this paper.

\section{Background}

Let $k$ be a field, $G$ a finite group and $V$ a finite-dimensional $kG$-module. Let $v_1,v_2, \ldots, v_n$ be a basis of $V$ and let $x_1,x_2, \ldots, x_n$ be the corresponding dual basis of $V^*$. The action of $G$ on $V$ induces an action of $G$ on $V^*$ defined by
\[(g(f))(v) = f(g^{-1}(v))\]
and we may extend this, by algebra automorphisms, to an action on the $k$-algebra $k[V]:= S(V^*)$. Setting $\deg(x_i)=2$ for each $i$ endows $S(V^*)$ with a grading which is respected by the $G$-action. The fixed points $k[V]^G$ of this action form a graded subalgebra called the subalgebra of invariants.

Now let $W$ be a further $kG$-module with basis $w_1, w_2, \ldots, w_m$. Let $k[V] \otimes W$ denote the free $k[V]$-module spanned by this basis. $G$ acts on $k[V] \otimes W$ via the formula
\[g(f \otimes w_i) = g (f) \otimes g (w_i)\] for $f \in k[V]$, extended linearly, and one may now show that the fixed points $(k[V] \otimes W)^G$ of this action form a $k[V]^G$-module. This is called the module of $W$-covariants. Since $k[V]$ is finitely generated over $k[V]^G$, the same is true of $(k[V] \otimes W)$. Since $k[V]^G$ is Noetherian and $(k[V] \otimes W)^G$ is a submodule of $k[V] \otimes W$, we get that the module $W$-covariants is finitely generated in general.

Consider the special case $R_i(V):= k[V] \otimes (\Lambda^i(V^*))$. This is called the module of differential $i$-forms. Usually one writes $dx_1,dx_2, \ldots, dx_n$ for the $k[V]$-basis of $R_i(V)$, which explains the terminology. The exterior product $\wedge: \Lambda^i(V^*) \times \Lambda^{j}(V^*) \rightarrow \Lambda^{i+j}(V^*)$ can be extended $k[V]$-linearly to an exterior product on $R(V):= \bigoplus_{i=0}^n R_i(V) = S(V^*) \otimes \Lambda(V^*)$ (note that $\Lambda^i(V) = 0$ for $i>n)$ which makes $R(V)$ into a non-commutative $k$-algebra. The fixed points $R(V)^G$ are called the algebra of differential invariants. 

 Another way of viewing $R(V)$ is as follows: write $y_i:= dx_i$ for each $i$ and endow $\Lambda^*(V^*) = \Lambda[y_1,y_2, \ldots, y_n]$ with a grading by setting $\deg(y_i)=1$ for each $i$. Then $R(V)$ is a commutative-graded $k$-algebra, i.e. for each $f, g \in R$ we have
\[fg = (-1)^{\deg(f)\deg(g)}gf.\]

The goal of this paper is to find a minimal generating set for $R(M)^G$, where $M$ is the vector space of $2 \times 2$ traceless matrices over $\F_3$ and the action of $G=\SL_2(F_3)$ is by conjugation. In order to determine the module generators of the summands $R_i(M)$, it is helpful to recall the following notions.\\

 For $S$ a finitely generated $k$-algebra,  $k$ a field, and $B$ a finitely generated $S$-module, a homogeneous system of parameters of positive degree for $B$ is a set of elements $\{f_1,...,f_n\} \subset S$ where $n$ is the Krull-dimension of $B$, with the property that $B$ is finitely generated as a module over the ring $A=k[f_1,...,f_n]$. It can be shown that if $B$ a free module over a subalgebra of $S$ generated by a homogeneous system of parameters, it is free over any such subalgebra. We say $B$ is a Cohen-Macaulay $S$-module if it has this property. In addition, $S$ is a Cohen-Macaulay algebra if it is a Cohen-Macaulay module over itself.\\

A further fundamental concept used in this paper is a version of Molien's formula for the Hilbert series of a ring of relative invariants. For $S$ a finitely generated $k$-algebra with homogeneous components $S_d$, the Hilbert series of $R$ is the power series
\begin{equation}
    H(S,t)=\sum_{0\leq i} \dim_k(S_i)t^i.
\end{equation}
Let $G$ be a finite group, $V$ a vector space and $k$ a field with char$(k)=0$. Further let $\chi: G \rightarrow k$ be a linear character. We say $f \in k[V]$ is a $\chi$-relative invariant if $g(f) = \chi(g)f$ for all $g \in G$. The set of $\chi$-relative $G$-invariants is denoted $k[V]^G_\chi$ and is a module over $k[V]^G$. Molien's formula (found for example in \cite[Theorem~3.7.1]{CW}) was established for $k$ such that char$(k)=0$. In fact we will need a version of this result in a finite field of order $p$ not dividing $|G|$. We proceed as follows: let $e$ be the exponent of $G$. Let $K$ be a field of characteristic $0$ containing a primitive $e$th root of unity and choose a homomorphism $\phi: k \rightarrow K$ which identifies $k$ with the multiplicative group of $e$th roots of unity in $K$. For any linear character $\chi: G \rightarrow k$ we define the Brauer lift $\chi^0: G \rightarrow K$ by $\chi^0 = \phi \circ \chi$. Further, for any matrix $X \in \GL_n(k)$ we define $\det^0(X) \in K$ to be the product of $\phi(\lambda)$ as $\lambda$ runs through the eigenvalues of $X$. The following formula is Theorem 2.5.2 in \cite{B}
\begin{equation}\label{molien}
    H(k[V]_\chi^G,t)=\frac{1}{|G|}\big(\sum_{g\in G}\frac{\chi(g)}{\textnormal{det}(Id_V-t g)} \big),
\end{equation}
and it holds for $k$ a field of characteristic zero. As observed in \cite{B}, the formula holds when $k$ is a field of finite characteristic not dividing $|G|$ after replacing $\chi$ and $\det$ by$\chi^0$ and $\det^0$ respectively.

The computation of the generators of the summands in $R_i(M)^G$ relies on some material in the study of covariants. In particular, we will require the following fundamental result, due to Eagon and Hochster \cite{EH}

\begin{Lemma}\label{EH} Let $k$ be a field whose characteristic does not divide $|G|$. Then $(k[V] \otimes W)^G$ is a Cohen-Macaulay $k[V]^G$-module.
\end{Lemma}

\noindent Let $H\leq G$, $T$ be a left-transversal of $H$ in $G$, and $\{e_t|t\in T\}$ be a basis for the left permutation module $k[G/H]$. Recall that $H$ acts trivially on this. We define a map
\begin{center}
$\Theta:k[V]^H \rightarrow k[V]\otimes k(G/H)$\\
$\Theta:f \mapsto \Sigma_{t\in T}tf\otimes e_t$.
\end{center}
\noindent The following result on covariants is a minor generalisation of Proposition 5 in \cite{E} and its proof is crucial for our computation:

\begin{Lemma}\label{use molien}
Let $k$ be a field, $G$ a group, and $\mathcal{A}$ a $k G$-algebra. Let $H \leq G$ and let $T$ be a set of left coset representatives for $H$ in $G$. Then the map $\Theta: \mathcal{A}^H \rightarrow (\mathcal{A} \otimes k(G/H))^G$ defined by
\[\Theta(f) = \sum_t t ( f \otimes t)\] is an isomorphism of $\mathcal{A}^G$-modules.
\end{Lemma}

\begin{proof} An adaptation of the proof of \cite[Proposition~5]{E} with $\mathcal{A}$ replacing $k[V]$ throughout.
\end{proof}

\noindent We shall continue with the computation of $(S(M^*)\otimes \Lambda (M^*))^{SL_2(\F_3)}$.

\section{Computing Invariants}

It is convenient, for reasons which will become clear later, to choose a basis
\[v_1 = \begin{pmatrix} 0 & 1 \\ -1 & 0 \end{pmatrix}, v_2 =  \begin{pmatrix} -1 & -1 \\ -1 & 1 \end{pmatrix}, v_3 =  \begin{pmatrix} 1 & -1 \\ -1 & -1 \end{pmatrix} \]
for $M$. Let $x_1,x_2,x_3$ be the corresponding dual basis for $M^*$, which will generate $S(M^*)$ as a commutative algebra. Let $y_1,y_2,y_3$ be the dual basis of generators for $\Lambda(M^*)$. Now $R$ can be viewed as a free graded-commutative algebra generated over $\F_3$ by $x_1,x_2,x_3,y_1,y_2,y_3$ in which $\deg(x_i)=2$ and $\deg(y_i)=1$. We define an action of $G$ on $M$ by setting
\[g( v) = gvg^{-1}\]
for $g \in G$ and $v \in M$. The induced action on $M^*$ is then denoted by $g(f)$ for $g \in G$ and $f \in M^*$, and this extends algebraically to an action on $R$.

We label certain elements of $G$: let 
\[t:= \begin{pmatrix} 1 & 1\\ 0 & 1 \end{pmatrix},\]
which generates the Sylow-3-subgroup $P \cong C_3$ of $G$. Let
\[i:= \begin{pmatrix} 0 & 1 \\ -1 & 0 \end{pmatrix}.\]
Note that $t$ and $i$ together generate $G$. Further, the kernel $K$ of the action of $G$ on $M$ is $\{\pm e\} = \{e, i^2\}$ where $e$ denotes the identity matrix. 
Further, we set
\[j:= t^{-1}it; k:= tit^{-1}.\]
Notice that $i^2=j^2=k^2=-e$ and $ij=k, jk=i, ki=j$; in other words, $i, j$ generate a subgroup $H = \{\pm e, \pm i, \pm j, \pm k\} \cong Q_8$. Moreover, the matrices representing $j$ and $k$ are the same as $v_3$ and $v_2$ respectively, from which we see immediately that
\[t ( v_s) = v_{s+1}\] with subscripts understood modulo 3. 
Let $\rho: G \rightarrow \GL_3(\F_3)$ be the representation afforded by the action of $G$ on $M$. Then we have
\[\rho(t) = \begin{pmatrix} 0 & 1 & 0\\ 0 & 0 & 1\\ 1 & 0 & 0 \end{pmatrix}\]
and
\[\rho(i) = \begin{pmatrix} 1 & 0 & 0\\ 0 & -1 & 0 \\ 0 & 0 & -1 \end{pmatrix}.\]
Each of these is their own inverse-transpose, therefore
\[t( x_s) = x_{s+1},\] ditto $y_s$. Similarly one may check that
\[i( x_1) = x_1, i ( x_2) = -x_2, i ( x_3) = -x_3,\] ditto $y_s$ again.
For future use we note that
\[j ( x_1) = -x_1, j ( x_2) = -x_2, i ( x_3) = x_3,\]
and
\[k ( x_1) = -x_1, k ( x_2) = x_2, i ( x_3) = -x_3.\]
Define a one-dimensional representation $W = \langle w \rangle$ of $H$ as follows:
\[i ( w) = w, j ( w=)-w, k ( w)=-w.\]
Let $\chi$ denote the corresponding character $H \rightarrow \F_3^{\times}$. The following observation is crucial:

\begin{Lemma}
We have $M^* \cong \F_3(G/H) \otimes W$.
\end{Lemma}

\begin{proof} 
$\{e,t,t^2\}$ is a set of left coset representatives for $H$ in $G$. Define $\phi: M^* \rightarrow  \F_3(G/H) \otimes W$ as follows:
\[ \phi(x_s) = t^{s-1} \otimes w.\]
Then $$t ( \phi(x_s)) = t^s \otimes w = \phi(x_{s+1}) = \phi(t ( x_s))$$
and
\begin{align*}
i ( \phi(x_1)) &= i ( (e \otimes w)) &= e \otimes (i (w)) &= e \otimes w &= \phi(x_1) &= \phi(i ( x_1));\\
i ( \phi(x_2)) &= i ( (t \otimes w)) &= t \otimes (j ( w)) &= -t \otimes w &= -\phi(x_2) &= \phi(i ( x_2));\\
i ( \phi(x_3)) &= i ( (t^2 \otimes w)) &= t^2 \otimes (k ( w)) &= -t^2 \otimes w &= -\phi(x_3) &= \phi(i ( x_3)).
\end{align*}
So $\phi$ is an isomorphism.
\end{proof}

The ring of invariants $S(M^*_p)^{G_p}$, where $M_p$ denotes the $2 \times 2$ trace-free matrices over $\F_p$ and $G_p = \SL_2(\F_p)$, was computed by Anghel in \cite{An}. Translating his result for $p=3$ into our preferred basis, we obtain:

\begin{prop}
Let 
\begin{align*} a_1 &= x_1^2+x_2^2+x_3^2;\\
a_2 &= x_1x_2x_3\\
a_3 &= x_1^4+x_2^4+x_3^4
\end{align*}

The algebra of invariants $S(M^*)^G$ is Cohen-Macaulay, and a free module over $A:=K[a_1,a_2,a_3]$ generated by $1$ and
\[b: = x_1^4x_2^2+x_1^2x_3^4+x_2^4x_3^2.\]
\end{prop} 

\begin{rem}
As $\dim(M^P)=1 \geq \dim(M)-2$ we get that $S(M^*)^G$ is Cohen-Macaulay by \cite{FKS}. It is easy to see directly that $\{a_1,a_2,a_3\}$ form a homogeneous system of parameters. Now by Theorem 3.7.1 in \cite{DK}  we get that $S(M^*)^G$ has rank 2 over $A$ (remember that the action of $G$ is not faithful, it is really an action of $G/K$ which has order 12). We see that $\{1,b\}$ are the secondary invariants by computing the invariants in all degrees up to six. 
\end{rem}


We recall that $\Lambda$ is a graded algebra consisting of vector spaces and as such we can decompose $R=S(M^*) \otimes \Lambda(M^*)$ as
\begin{center}
    $R=\bigoplus_{i=0}^3 S(x_1,x_2,x_3)\otimes \Lambda^i(y_1,y_2,y_3)$.
\end{center}
 We shall compute the invariants for each of these summands individually.\\
 
The actions of $G$ on $\Lambda^3(M^*) = \langle y_1y_2y_3 \rangle$ and $\Lambda^0(M^*) = \langle 1 \rangle$ are trivial. We obtain immediately:

\begin{prop}\
\begin{itemize}
\item[(a)] $(S(M^*) \otimes \Lambda^3(M^*))^G$ is a free $A$-module generated by $y_1y_2y_3$ and $by_1y_2y_3$;
\item[(b)] $(S(M^*) \otimes \Lambda^0(M^*))^G$ is a free $A$-module generated by $1$ and $b$.
\end{itemize}
\end{prop}

On the other hand we have 
\begin{center}
    $\Lambda^1(M^*) = \langle y_1,y_2,y_3 \rangle \cong M^*$ and $\Lambda^2(M^*) = \langle y_2y_3,y_3y_1,y_1y_2 \rangle \cong M^*$.
\end{center} 

Applying Lemma \ref{use molien} to the above with $H = Q_8$, $G = \SL_2(F)$ and $\mathcal{A} = S(M^*) \otimes W$ yields a pair of $S(M^*)^G$-module isomorphisms $\Theta_i: (S(M^*) \otimes W)^H \rightarrow (S(M^*) \otimes \Lambda^i(M^*))^G$ for $i=1,2$, defined by
\begin{align*} \Theta_1(f) &=  \sum_{j=0}^2 t^j ( f \otimes y_{j+1})\\
\Theta_2(f) &= \sum_{j=0}^2 t^j ( f \otimes y_{j+2}y_{j+3})
\end{align*}
with indices understood modulo 3.

Note that elements of $(S(M^*) \otimes W)^H$ are precisely those $f \in S(M^*)$ satisfying $h ( f) = \chi(h) f$. In other words, they are the relative invariants $S(M^*)^H_\chi$. As $|H| \neq 0 \mod 3$, $S(M^*)^H_{\chi}$ is a Cohen-Macaulay $S(M^*)^H$-module, by Lemma \ref{EH}. It cannot be a free $S(M^*)^H$-module, by \cite{R}, since $H$ is not a reflection group. However, since $(S(M^*) \otimes W)^H$ is a finitely generated $S(M^*)^H$ module, which is in turn a fintely generated $A$-module, $(S(M^*) \otimes W)^H$ is finitely generated over $A$. Since it is Cohen-Macaulay, it is therefore a free $A$-module. We have therefore proved that $S(M^*) \otimes \Lambda^i(M^*))^G$ is a free $A$-module, for $i=1,2$. In particular, this shows that $R(M)$ is a Cohen-Macaulay $S(M^*)^G$-module.

It remains to find generators for $S(M^*)^H_\chi$. Note once more that the action of $H$ is not faithful as $K \leq H$. Write $\overline{H} = H/K$. As $\overline{H}$ is non-modular we can use our adaptation of \eqref{molien} to compute the Hilbert Series of $S(M^*)^H_\chi$ . We find:

\begin{align*} H(S(M^*)_\chi^H, t) &= \frac14 \sum_{\sigma \in \overline{H}} \frac{\chi^0(\sigma)}{\det^0(\id_M-\rho(\sigma) t)}\\
&=\frac14 \left( \frac{1}{(1-t)^3} + \frac{1}{(1-t)(1+t)^2} + 2\frac{-1}{(1-t)(1+t)^2} \right)\\
&= \frac14 \left(\frac{1}{(1-t)^3}-\frac{1}{(1-t)(1+t^2)} \right)\\
&= \frac14 \left(\frac{(1+t)^2-(1-t)^2}{(1-t)^3(1+t)^2} \right)\\
&= \frac{t}{(1-t)^3(1+t)^2} \\
&=  \frac{t}{(1-t)(1-t^2)^2} \\
&= \frac{t+t^2}{(1-t^2)^3}
\end{align*}

A homogeneous system of parameters for $S(M^*)^H$ is given by $\{x_1^2,x_2^2,x_3^2\}$, and the above shows that $S(M^*)^H_{\chi}$ is a free module over the subalgebra $F[x_1^2,x_2^2,x_3^2]$ with generators in degrees 1 and 2, which are easily seen to be $x_1$ and $x_2x_3$. However, we need generators over $A$, so we rewrite the series as
\begin{align*} \frac{(t+t^2)(1+t^2)}{(1-t^2)(1-t^2)(1-t^4)}
&= \frac{(t+t^3)}{(1-t^2)(1-t)(1-t^4)}\\
&=\frac{(t+t^3)(1+t+t^2)}{(1-t^2)(1-t^3)(1-t^4)}\\
&= \frac{(t+t^2+2t^3+t^4+t^5)}{(1-t^2)(1-t^3)(1-t^4)}.
\end{align*}

Thus, $S(M^*)^H_{\chi}$ is a free $A$-module with generators in degrees 1,2,3 (twice), 4 and 5.

\begin{prop}
$S(M^*)^H_{\chi}$ is generated freely over $A$ by the set 
$$\{x_1,x_2x_3,x_1^3,x_1x_2^2,x_2^3x_3,x_1^3x_2^2\}.$$
\end{prop}

\begin{proof}
It is easy to see each of the given monomials lies in $S(M^*)^H_{\chi}$. Let $B$ be the $A$-module generated by the given set. By the previous paragraph, the elements of $S(M^*)^H_{\chi}$ are sums of monomials of the form $m =x_1^ix_2^jx_3^k$ where $i \neq j=k \mod 2$. We will prove by induction on $d:=\deg(m)$ that every such monomial belongs to $B$.

If $d=1$, then $m=x_1$ and we are done. If $d=2$, then $m = x_2x_3$ and we are done. 

If $d=3$ then $m = x_1^3, m=  x_1x_2^2$ or $m = x_1x_3^2$. In the first two cases we are done, and in the third we have
\[x_1x_3^2 = a_1x_1 - x_1^3 - x_1x_2^2 \in B\] as desired.

If $d=4$ then $m = x_1^2x_2x_3$,  $m= x_2^3x_3$ or $m=x_2x_3^2$. In the first case we have
\[m = a_2x_1 \in B.\] In the second case we are done, and in the third we have
\[x_2x_3^3 = a_1x_2x_3-a_2x_1-x_2^3x_3 \in B\] as desired.

If $d=5$ then $m \in \{x_1^5, x_1^3x_2^2, x_1^3x_3^2,x_1x_2^4,x_1x_2^2x_3^2,x_1x_3^4\}$. 
We have $x_1^3x_2^2 \in B$ and $x_1x_2^2x_3^2 = a_2x_2x_3 \in B$. Now we have
\[a_1x_1x_2^2 = x_1^3x_2^2+x_1x_2^4+x_1x_2^2x_3^2\] which shows that $x_1x_2^4 \in B$. 
Further,
\[a_1^2x_1-a_3x_1 = -x_1^3x_2^2-x_1^3x_3^2-x_1x_2^2x_3^2\] which shows that $x_1^3x_3^2 \in B$.
Now
\[a_1x_1^3 = x_1^5+x_1^3x_2^2+x_1^3x_3^2\] which shows that $x_1^5 \in B$.
And finally,
\[a_3x_1 = x_1^5+x_1x_2^4+x_1x_3^4\] which shows that $x_1x_3^4 \in B$ as required.

Now assume that $d \geq 6$ and assume the claim is proven for all monomials of degree $<d$.
Suppose first that $d$ is even, then we have
$m = x_1^{2i}x_2^{2j+1}x_3^{2k+1}$ for some $i,j,k \geq 0$ with $i+j+k \geq 2$. If $i \geq 1$ we may write
\[m = a_2x_1^{2i-1}x_2^{2j}x_3^{2k} \in B\] by induction, so we may assume $i=0$ and $j+k \geq 2$.
Let $j \geq 1$, then we have
\begin{align*} m &= a_1x_2^{2j-1}x_3^{2k+1}-x_1^{2}x_2^{2j-1}x_3^{2k+1}-x_2^{2j-1}x_3^{2k+3}\\
&= a_1x_2^{2j-1}x_3^{2k+1}-a_2x_1x_2^{2j-2}x_3^{2k}-x_2^{2j-1}x_3^{2k+3}
\end{align*}
By induction, the first two terms lie in $B$ and we have
\begin{equation}\label{rel} x_2^{2j+1}x_3^{2k+1}= -x_2^{2j-1}x_3^{2k+3} \mod B\end{equation} for all $j \geq 1$.
Iterating the argument shows that
\begin{equation}\label{rel2} x_2^{2j+1}x_3^{2k+1}= x_2^{2j-3}x_3^{2k+5} \mod B\end{equation} for all $j \geq 2$.
However, we also have
\begin{align*} m &= a_3x_2^{2j-3}x_3^{2k+1}-x_1^{4}x_2^{2j-3}x_3^{2k+1}-x_2^{2j-3}x_3^{2k+5}\\
&= a_3x_2^{2j-3}x_3^{2k+1}-a_2x_1^{3}x_2^{2j-4}x_3^{2k}-x_2^{2j-3}x_3^{2k+5}
\end{align*}
which shows that \begin{equation} x_2^{2j+1}x_3^{2k+1}= -x_2^{2j-3}x_3^{2k+5} \mod B\end{equation} for all $j \geq 2$. By \eqref{rel} and \eqref{rel2} we must have $$x_2^{2j-3}x_3^{2k+5} = x_2^{2j-1}x_3^{2k+3} = x_2^{2j+1}x_3^{2k+1} = 0 \mod B$$ for all $j \geq 2, k \geq 0$ as required.

Now suppose that $d$ is odd, so $d \geq 7$ and $m = x_1^{2i+1}x_2^{2j}x_3^{2k}$ for some $i,j,k \geq 0$ with $i+j+k \geq 3$. If $j \geq 1$ and $k \geq 1$ then we may write
\[m = a_2x_1^{2i}x_2^{2j-1}x_3^{2k-1} \in B\]
by induction. So we may assume $j=0$ or $k=0$. Assume $k=0$, then $i+j \geq 3$, and we have, for all $i, j \geq 1$,
\begin{align*} m &= a_1x_1^{2i-1}x_2^{2j}-x_1^{2i-1}x_2^{2j+2}-x_1^{2i-1}x_2^{2j}x_3^{2}\\
&= a_1x_1^{2i-1}x_2^{2j}-x_1^{2i-1}x_2^{2j+2}-a_2x_1^{2i-2}x_2^{2j-1}x_3
\end{align*}
By induction, the first and last terms lie in $B$ and we have
\begin{equation}\label{oddrel} x_1^{2i+1}x_2^{2j}= -x_1^{2i-1}x_2^{2j+2} \mod B\end{equation} for all $i, j \geq 1$.
Iterating the argument shows that
\begin{equation}\label{oddrel2} x_1^{2i+1}x_2^{2j}= x_1^{2j-3}x_2{2j+4} \mod B\end{equation} for all $i \geq 2, j \geq 1$.
However, we also have
\begin{align*} m &= a_3x_1^{2i-3}x_2^{2j}-x_1^{2i-3}x_2^{2j+4}-x_1^{2i-3}x_2^{2j}x_3^{2}\\
&= a_1x_1^{2i-3}x_2^{2j}-x_1^{2i-3}x_2^{2j+4}-a_2x_1^{2i-4}x_2^{2j-1}x_3
\end{align*}
for all $i \geq 2, j \geq 1$ which shows that 
 \begin{equation} x_1^{2i+1}x_2^{2j}= -x_1^{2i-3}x_2^{2j+4} \mod B.\end{equation}
By \eqref{oddrel} and \eqref{oddrel2} we must have
 \begin{equation}\label{j} x_1^{2i+1}x_2^{2j}= x_1^{2i-1}x_2^{2j+2}=x_1^{2i-3}x_2^{2j+4} = 0 \mod B\end{equation} for all $i \geq 2$ and $j \geq 1$. A similar argument shows that 
 \begin{equation}\label{k} x_1^{2i+1}x_3^{2k}= x_1^{2i-1}x_3^{2k+2}=x_1^{2i-3}x_3^{2k+4} = 0 \mod B\end{equation} for all $i \geq 2$ and $k \geq 1$. It remains to show that $x_1^d \in B$: we have
\[x_1^d = a_1(x_1^{d-2}) - x_1^{d-2}x_2^2 - x_1^{d-2}x_3^2.\]
The first term lies in $B$ by induction, the others by \eqref{j} and \eqref{k} above. This completes the induction and the claim is proved.

Therefore $B = S(M^*)^H_{\chi}$ as claimed. Further, the generating set must be free, because the Hilbert series of $B$ is bounded above by the Hilbert series of $ S(M^*)^H_{\chi}$ with equality if and only if the generating set is free.
\end{proof}

\begin{Corollary}\
\begin{itemize}
\item[(a)] $(S(M^*) \otimes \Lambda^1(M^*))^G$ is generated as an $A$-module by the set
\begin{align*} \{c_1 &= x_1y_1+x_2y_2+x_3y_3,\\
c_2 &= x_2x_3y_1+x_3x_1y_2+x_1x_2y_3,\\
c_3 &= x_1^3y_1+x_2^3y_2+x_3^3y_3,\\
c_4 &=x_1x_2^2y_1+x_2x_3^2y_2+x_3x_1^2y_3,\\
c_5 &= x_2^3x_3y_1+x_3^3x_1y_2+x_1^3x_2y_3,\\
c_6 &= x_1^3x_2^2y_1+x_2^3x_3^2y_2+x_3^3x_1^2y_3\}.\end{align*}
\item[(b)]  $(S(M^*) \otimes \Lambda^2(M^*))^G$ is generated as an $A$-module by the set
\begin{align*} \{d_1 &= x_1y_2y_3+x_2y_3y_1+x_3y_1y_2,\\ 
d_2 &= x_2x_3y_2y_3+x_3x_1y_3y_1+x_1x_2y_1y_2,\\
d_3 &= x_1^3y_2y_3+x_2^3y_3y_1+x_3^3y_1y_2,\\
d_4 &=x_1x_2^2y_2y_3+x_2x_3^2y_3y_1+x_3x_1^2y_1y_2,\\
d_5 &= x_2^3x_3y_2y_3+x_3^3x_1y_3y_1+x_1^3x_2y_1y_2,\\
d_6 &= x_1^3x_2^2y_2y_3+x_2^3x_3^2y_3y_1+x_3^3x_1^2y_1y_2\}.\end{align*}
\end{itemize}
\end{Corollary}

\begin{Theorem}\label{Final}
$R^G$ is minimally generated as a commutative-graded algebra by
\[\{a_1,a_2,a_3,b,c_1,\ldots, c_6, d_1,d_2,d_3,y_1y_2y_3\}.\] 
\end{Theorem}

\begin{proof}
 Obviously 
\[\{a_1,a_2,a_3,b,c_1,\ldots, c_6, d_1, \ldots, d_6,y_1y_2y_3\}\]  is a generating set. One may check the following relations:
\begin{align*} c_1c_2 &= -d_4-a_1d_1+d_3,\\
c_1c_3 &= a_1d_2+d_5-a_2d_1,\\
c_2c_3 &= -d_6+a_1d_4-a_1d_3+a_2d_2-a_3d_1-a_1^2d_1.
\end{align*}
So $d_4,d_5$, and $d_6$ are obsolete. To see that the generating set cannot be pruned any further we consider $x$-
 and $y$- degrees separately. Each $c_i$ has $y$-degree 1 and $x$-degree $<6 = \deg(b)$, so it can only possibly be rewritten in terms of $a_1,a_2,a_3$ and $c_j$ where $\deg_x(c_j) \leq \deg_x(c_i)$, which is ruled out by freeness over $A$. Similarly each $d_i$ has $y$-degree 2 and $x$-degree $<6 = \deg(b)$, so it cannot be rewritten in terms of $a_1,a_2,a_3,b$ and $d_j$ where $\deg_x(d_j) \leq \deg_x(d_i)$ by freeness over $A$. 
Further, for $d_1$ and $d_2$, no product of the form $c_ic_j$ has small enough to $x$-degree to rewrite these (except for $c_1^2$ which is obviously zero). Finally, the first relation above must be the unique expression of $c_1c_2$ a in terms of $A(d_1,d_2,d_3,d_4,d_5,d_6)$ by freeness, so we cannot write $d_3$ in terms of $c_1c_2$ and $A(d_1,d_2)$. Since $\deg_x(d_3) = 3$, $c_1c_2$ is the only product of generators with $y$-degree one which could be a summand of $d_3$.  
\end{proof}

\vspace{2cm}

\bibliographystyle{amsplain}

\begin{thebibliography}{10}




\bibitem {AM} Alejandro ADEM, R. James MILGRAM: \textit{Cohomology of Finite Groups}. Springer Verlag Berlin Heidelberg, 2004

\bibitem {An} Nicolae ANGHEL: \textit{{${\rm SL}_2$}-polynomial invariance}. Romanian Journal of Pure and Applied Mathematics, volume 43, 1998

\bibitem {B} David J. BENSON: \textit{Polynomial invariants of finite groups}. Cambridge University Press, 1993

\bibitem {BC}  Abraham BROER, Jianjun CHUAI Chuai: \textit{ Modules of covariants in modular invariant theory}. Proc.
Lond. Math. Soc. (3), 100(3):705–735, 2010

\bibitem {CE} Henri CARTAN, Samuel EILENBERG: \textit{Homological Algebra}. Princeton University Press, 1956

\bibitem {CW} H.E.A. Eddy CAMPBELL, David L. WEHLAU: \textit{Modular Invariant Theory}. Springer Verlag Berlin-Heidelberg, 2011

\bibitem {DK} Harm DERKSEN Gregor KEMPER: \textit{Computing invariants of algebraic groups in arbitrary characteristic}. Advances in Mathematics, Volume 217, Issue 5, 2008

\bibitem{E} Jonathan ELMER: \textit{Cohen-Macaulay modules of covariants for cyclic $p$-groups}. Preprint, arXiv:2506.03677v1, 2025

\bibitem{EH} M. HOCHSTER and John A. EAGON: \textit{Cohen-Macaulay rings, invariant theory, and the generic
perfection of determinantal loci}. Amer. J. Math., 93:1020–1058, 1971.

\bibitem {FKS} Peter FLEISCHMANN, Gregor KEMPER and R. James SHANK: \textit{Depth and cohomological connectivity in modular invariant theory}. Trans. Amer. Math. Soc. 357, 3605-3621, 2005

\bibitem {M} Anja MEYER: \textit{On the modular cohomology of $GL(2,p^n)$ and $SL(2,p^n)$}. Preprint,  arXiv:2506.04720v1, 2025

\bibitem {R} Victor REINER: \textit{Free modules of relative invariants of finite groups}. Studies in Applied Mathematics, vol. 81, 181-184, 1989 

\bibitem {S} Richard G SWAN: \textit{The p-period of a finite group}. Ill. J. Math. 4, 341-346, 1960





\end{thebibliography}

\end{document}